\documentclass[12pt]{amsart}
\usepackage{graphicx,amscd, color,amsmath,amsfonts,amssymb,geometry,amssymb}
\newtheorem{theorem}{Theorem}[section]

\newtheorem{example}[theorem]{Example}

\newtheorem{lemma}[theorem]{Lemma}

\newtheorem{definition}[theorem]{Definition}
\numberwithin{equation}{section}

\begin{document}
\title{On very non-linear subsets of continuous functions}
\author[G. Botelho, D. Cariello, V.V. F\'{a}varo, D. Pellegrino and J.B. Seoane]{G. Botelho, D. Cariello, V.V. F\'{a}varo, D. Pellegrino and J.B. Seoane-Sep\'{u}lveda}

\address{Faculdade de Matem\'atica, \newline\indent Universidade Federal de Uberl\^{a}ndia, \newline\indent 38.400-902 Ð Uberl\^{a}ndia, Brazil.}\email{botelho@ufu.br, dcariello@famat.ufu.br, vvfavaro@gmail.com}

\address{Departamento de Matem\'{a}tica, \newline\indent Universidade Federal da Para\'{\i}ba, \newline\indent 58.051-900 - Jo\~{a}o Pessoa, Brazil.}\email{pellegrino@pq.cnpq.br}

\address{Departamento de An\'{a}lisis Matem\'{a}tico,\newline\indent Facultad de Ciencias Matem\'{a}ticas, \newline\indent Plaza de Ciencias 3, \newline\indent Universidad Complutense de Madrid,\newline\indent Madrid, 28040, Spain.}
\email{jseoane@mat.ucm.es}

\thanks{G. Botelho is supported by CNPq Grant 302177/2011-6 and Fapemig Grant PPM-00295-11; J.B. Seoane-Sep\'{u}lveda was supported by MTM2012-34341.}



\maketitle

\begin{abstract}
In this paper we continue the study initiated by Gurariy and Quarta in 2004 on the existence of linear spaces formed, up to the null vector, by continuous functions that attain the maximum only at one point. Inserting a topological flavor to the subject, we prove that results already known for functions defined on certain subsets of $\mathbb{R}$ are actually true for functions on quite general topological spaces. In the line of the original results of Gurariy and Quarta, we prove that, depending on the desired dimension, such subspaces may exist or not.
\end{abstract}

\section*{Introduction}

The problem of finding linear spaces formed, up to the origin, solely by real-valued continuous functions on certain subsets of $\mathbb{R}$ that attain the maximum only at one point was successfully investigated by
V.I. Gurariy and L. Quarta \cite{GQ}. Given a topological space $D$, by $\widehat{C}(D)$ we denote the subset of the linear space $C\left(  D\right)  $ of all real-valued continuous functions on $D$ composed by the functions that attain the maximum exactly once in $D$. The main results obtained by Gurariy and Quarta in this direction are the following:

\begin{enumerate}
\item[(A)] $\widehat{C}[a,b)$ contains, up to the origin, a $2$-dimensional linear subspace of $C[a,b)$.

\item[(B)] $\widehat{C}(\mathbb{R})$ contains, up to the origin, a $2$-dimensional linear subspace of $C(\mathbb{R})$.

\item[(C)] There is no $2$-dimensional linear subspace of $C\left[ a,b\right]  $ contained in $\widehat{C}\left[a,b\right]\cup \{0\}$.
\end{enumerate}

In the words of Gurariy and Quarta (see, e.g., \cite{GQ}): $\widehat{C}[a,b)$ and $\widehat{C}(\mathbb{R})$ are $2$-lineable and $\widehat{C}\left[a,b\right]$ is very non-linear (we refer the interested reader to \cite{AGS,GQ} for an account on  this recently coined concept of {\em lineability}).

The situation is interesting because, for example, although $\widehat{C}\left[  a,b\right]  $ is a dense $G_\delta$ subset of $C[a,b]$, it does not contain, up to the origin, a 2-dimensional subspace of $C[a,b]$. The proof, due to Gurariy (2004), of the fact that $\widehat{C}\left[  a,b\right]  $ is a dense $G_\delta$ set has never appeared, so we shall sketch it in the Appendix for the benefit of the reader (and as a tribute to V.I. Gurariy).

On the one hand, the results (A) and (B) above can be obtained in a fairly simple way. Indeed, for (A) just take $\sin \cdot $ and $\cos \cdot$ on $[0, 2\pi)$, whereas for (B) consider the two linearly independent functions $x(t), y(t)$ defined on $\mathbb{R}$ as
$$x(t) := \mu(t) \cos(4 \arctan(|t|))  \quad \text{and} \quad y(t) := \mu(t) \sin(4 \arctan(|t|)),$$
where $\mu$ is the real valued continuous function defined on $\mathbb{R}$ by
\begin{equation*}
\mu(t)=\left\{
\begin{tabular}{cl}
$e^t$ & if $t \le 0$,\\
$1$ & if $t \ge 0$.
\end{tabular}
\right.
\end{equation*}
Then, take the $2$-dimensional vector space given by $\displaystyle V = \text{span} \{x(t), y(t) \}.$ It can be seen, quite easily, that $V \subsetneq \widehat{\mathcal{C}}(\mathbb R) \cup \{0\}.$ On the other hand, result (C) above requires a series of highly technical lemmas (see \cite{GQ})

The purpose of this paper is to obtain far-reaching generalizations of the aforementioned results of Gurariy and Quarta. The idea is to consider spaces of functions defined on domains much more general than the original ones ($[a,b), \mathbb{R}, [a,b]$, respectively), for which the corresponding results hold true. Moreover, we investigate the existence of $n$-dimensional subspaces -- instead of 2-dimensional subspaces -- formed by functions that attain the maximum exactly once. The search for such general domains disclosed the topological nature of the problem. In this way, for each of the three results we ended up with general domains that replicate a topological property of the original domain. While Gurariy and Quarta \cite{GQ} used typical analytic techniques, the manifested nature of the problem led us to apply topological techniques, for example the Borsuk-Ulam theorem.

This paper is arranged as follows. In Section \ref{Sec1} we extend (A) to spaces of functions defined on topological spaces $D$ that can be continuously embedded onto some Euclidean sphere $S^n$. In Section \ref{Sec2} we extend (B) to spaces of functions defined on quite general topological spaces $D$ that include $\mathbb{R}$. In the two former cases we prove that $\widehat{C}(D)\cup \{0\}$ contains an $(n+1)$-dimensional subspace. In Section \ref{Sec3} we extend (B) to spaces of functions defined on compact subsets $K$ of $\mathbb{R}^m$. In this case we prove that $\widehat{C}(K) \cup \{0\}$ does not contain an $(m+1)$-dimensional subspace of $C(K)$ for every compact $K \subset \mathbb{R}^m$ but, on the other hand, there are compact sets $K \subset \mathbb{R}^m$ for which $\widehat{C}(K) \cup \{0\}$ contains an $m$-dimensional subspace of $C(K)$. In a final section we provide an example of a compact space $K$ for which there is an infinite dimensional subspace of $C(K)$ formed by functions that attain the maximum only at one point.

\section{Continuous functions on preimages of Euclidean spheres}
\label{Sec1}
In this section we show that in (A) the interval $[a,b)$ can be replaced by preimages $D$ of Euclidean spheres. Moreover, the dimension of the resulting subspace of $C(D)$ contained in $\widehat{C}(D)\cup \{0\}$ equals the dimension of the sphere plus 1. By $\langle \cdot, \cdot \rangle$ we denote the usual inner product in the Euclidean spaces.
\begin{theorem}
\label{aaa}Let $n\geq2$ be a positive integer and $D$ be a topological space
for which there is a continuous bijection from $D$ to $S^{n-1}.$ Then $\widehat{C}\left(  D\right)  $ contains, up to the origin, an $n$-dimensional linear subspace of $C(D)$.
\end{theorem}

\begin{proof}Let $\pi_{i}\colon S^{n-1}\longrightarrow\mathbb{R}$ be the projection on the component
$i=1,\ldots,n$ and let $G\colon D\longrightarrow S^{n-1}$ be a continuous bijection. We
first note that the maps $\pi_{i}$, $1\leq i\leq
n$, are linearly independent. In fact, consider a linear combination
$\sum_{i=1}^{n}a_{i}\pi_{i}$. If $a_{i}\neq0$ for some $i$, then $\left(\sum_{i=1}^{n}a_{i}^{2}\right)^{1/2}\neq0$. Let
$z=(a_{1},\ldots,a_{n})/\left(\sum_{i=1}^{n}a_{i}^{2}\right)^{1/2}\in S^{n-1}$. Then
\[
\sum_{i=1}^{n}a_{i}\pi_{i}(z)=\dfrac{\sum_{i=1}^{n}a_{i}^{2}}{\left(\sum
_{i=1}^{n}a_{i}^{2}\right)^{1/2}}\neq0.
\]
Now we shall prove that each nontrivial linear combination of the functions $\pi_{i}$,
$1\leq i\leq n$, has only one point of maximum. Consider $\sum_{i=1}^{n}a_{i}\pi_{i}$ with some $a_{i}\neq0$. For an arbitrary $y\in S^{n-1}$,
\[
\sum_{i=1}^{n}a_{i}\pi_{i}(y)= \langle a,y \rangle,
\]
where $a=(a_{1},\ldots,a_{n})$. It is well known that the function $y \mapsto \langle a,y \rangle$ attains its maximum in $z\in
S^{n-1}$ if, and only if,
\begin{equation}
z=\frac{(a_{1},\ldots,a_{n})}{\left(\sum_{i=1}^{n}a_{i}^{2}\right)^{1/2}}.\label{dez111}%
\end{equation}
Therefore the unique point of maximum of $\sum_{i=1}^{n}a_{i}\pi
_{i}\colon S^{n-1}\longrightarrow\mathbb{R}$ is $z$ as in (\ref{dez111}).

Now, consider the compositions%
\[
\pi_{i}\circ G\colon D\longrightarrow\mathbb{R}%
\]
for $i=1,\ldots,n.$ It is easy to see that set $\left\{  \pi_{i}\circ
G:i=1,\ldots,n\right\}  $ is linearly independent. Let $
h=\sum_{i=1}^{n}b_{i}\left(  \pi_{i}\circ G\right)$ be a nontrivial linear combination of the functions $\pi_{i}\circ G.$ Since $\sum_{i=1}^{n}b_{i}\pi_{i}$
attains its maximum at an unique point $x_{0}\in S^{n-1}$, and $G$ is a
bijection, it follows that $h$ attains its maximum at an unique point, namely,
$G^{-1}\left(  x_{0}\right)  .$ The linear subspace spanned by the functions $  \pi_{i}\circ G, i= 1, \ldots, n,$
completes the proof.
\end{proof}

\section{Generalizing the 2-lineability of $\widehat{C}(\mathbb{R})$}\label{Sec2}

In this section we show that the argument of the proof of Theorem \ref{aaa} actually holds for more general domains; general enough to have $\mathbb{R}$ as a particular instance. By $\|\cdot\|_2$ we mean the Euclidean norm on $\mathbb{R}^n$.

\begin{theorem}
\label{bbbb}Let $n\geq2$ be a positive integer and $D$ be a topological space
containing a closed set $Y$ such that there are a continuous bijection $F \colon Y \longrightarrow S^{n-1}$ and a continuous extension $G \colon D \longrightarrow \mathbb{R}^n$ of $F$ such that $\|G(x)\|_2 < 1$ for every $x \notin Y$. Then $\widehat{C}\left(  D\right)  $ contains, up to the origin, an $n$-dimensional linear subspace of $C(D)$.
\end{theorem}

\begin{proof} The symbol $\pi_i$ stands for the $i$-th projection on $\mathbb{R}^n$.
Note that $S^{n-1}\subset G(D)$ and $G(D)$ is contained in the closed unit
ball of $\mathbb{R}^{n}$. Let us see that any nontrivial linear combination
\[
f:=\sum_{i=1}^{n}b_{i}\pi_{i}\colon G(D)\longrightarrow\mathbb{R}%
\]
attains its maximum at an unique point $x_{0}\in G(D)$ and that this point belongs
to $S^{n-1}.$ Indeed, restricting $\sum_{i=1}^{n}b_{i}\pi_{i}$ to
$S^{n-1}\subset G(D),$ the same argument of the previous section tells us that
there is an unique $x_{0}\in S^{n-1}$ such that%
\begin{equation}
\sum_{i=1}^{n}b_{i}\pi_{i}(x)\leq\sum_{i=1}^{n}b_{i}\pi_{i}(x_{0})\label{ssa1}%
\end{equation}
for every $x\in S^{n-1}$. So, for any $x \in S^{n-1}$, as $- x \in S^{n-1}$, we have
\[
\sum_{i=1}^{n}b_{i}\pi_{i}(-x)\leq\sum_{i=1}^{n}b_{i}\pi_{i}(x_{0})
\]
hence%
\begin{equation}
-\sum_{i=1}^{n}b_{i}\pi_{i}(x)\leq\sum_{i=1}^{n}b_{i}\pi_{i}(x_{0}%
).\label{ssa}%
\end{equation}
From (\ref{ssa1}) and (\ref{ssa}) we obtain
\begin{equation}
\left\vert \sum_{i=1}^{n}b_{i}\pi_{i}(x)\right\vert \leq\sum_{i=1}^{n}b_{i}%
\pi_{i}(x_{0})\label{ssa2}%
\end{equation}
for every $x\in S^{n-1}.$ Now we just need to show that if $y\in G(D)- S^{n-1}$, then
\begin{equation*}
\sum_{i=1}^{n}b_{i}\pi_{i}(y)<\sum_{i=1}^{n}b_{i}\pi_{i}(x_{0}).
\end{equation*}
The case $y=0$ is immediate. Let us suppose $y\neq0$. Recalling that $\left\Vert y\right\Vert
<1$ and applying (\ref{ssa2}), we have
\begin{align*}\sum_{i=1}^{n}b_{i}\pi_{i}(y)&\leq\left\vert \sum_{i=1}^{n}b_{i}\pi
_{i}(y)\right\vert =\left\Vert y\right\Vert \cdot \left\vert \sum_{i=1}^{n}b_{i}%
\pi_{i}\left(\frac{y}{\left\Vert y\right\Vert }\right)\right\vert\\& <\left\vert \sum
_{i=1}^{n}b_{i}\pi_{i}\left(\frac{y}{\left\Vert y\right\Vert }\right)\right\vert \leq
\sum_{i=1}^{n}b_{i}\pi_{i}(x_{0}).
\end{align*}
Thus far we have proved that any nontrivial linear combination 
$$f=\sum_{i=1}^{n}b_{i}\pi_{i}\colon G(D)\longrightarrow\mathbb{R}$$
attains its maximum at an unique point $x_{0}\in G(D)$ and that this point belongs to $S^{n-1}.$ Now we claim that any nontrivial linear combination
\[
h: =\sum_{i=1}^{n}b_{i}\left(  \pi_{i}\circ G\right)  \colon D\longrightarrow\mathbb{R}%
\]
attains its maximum only at the point $F^{-1}\left(  x_{0}\right)  \in S^{n-1}.$ In
fact, note that $h$ attains a maximum at a point $z\in D$ if and only if $f$
attains a maximum at $G(z)$. So, the only point of maximum of $h$ is $G\left(
F^{-1}\left(  x_{0}\right)  \right)  =x_{0}.$ Since $S^{n-1}\subset G(D)$ and the functions $\pi_i $, $i=1, \ldots, n$, are linearly independent on $S^{n-1}$, it follows that the functions $\pi_i \circ G$, $i=1,\ldots,n$, are linearly independent on $D$. The linear subspace spanned by the functions $  \pi_{i}\circ G, i= 1, \ldots, n,$
completes the proof.
\end{proof}

Let us see that Theorem \ref{bbbb} recovers the 2-lineability of $\widehat{C}(\mathbb{R})$. For $a,b \in \mathbb{R}^2$, by $(a,b)$ we mean the open line segment in $\mathbb{R}^2$ from $a$ to $b$. Take $Y = [0, +\infty) \subset \mathbb{R}$, a continuous bijection $F \colon [0, +\infty) \longrightarrow S^1$, a homeomorphism $g \colon (-\infty, 0) \longrightarrow (F(0), (0,0)) \subset \mathbb{R}^2$ such that $\lim\limits_{x \rightarrow 0}g(x) = F(0)$. Define
\[G \colon \mathbb{R} \longrightarrow \mathbb{R}^2~,~G(x) =
\left\{
\begin{array}
[c]{c}%
F(x), ~ x \geq 0,\\
g(x), ~ x < 0.
\end{array}
\right.
\]

\section{Functions on compact subsets of Euclidean spaces} \label{Sec3}

As to the result (C) of the Introduction on $\widehat{C}[a,b]$, in this section we prove that, on the one hand, $\widehat{C}(S^{m-1})\cup \{0\}$ contains an $m$-dimensional subspace of $C(S^{m-1})$; and on the other hand, for every compact subset $K$ of $\mathbb{R}^m$, $\widehat{C}(K)\cup \{0\}$ does not contain an $(m+1)$-dimensional subspace of $C(K)$.

First we have to prove four technical lemmas which will be used in the proof of the main result of this section (Theorem \ref{kkkp}).

\begin{lemma}
\label{lemmafirst} Let $D$ be a metric space such that there is a linear space
$V$ of continuous functions from $D$ to $\mathbb{R}$ such that

\begin{enumerate}
\item $\dim(V)=n$, $n \in \mathbb{N}$,

\item Each $0\neq f\in V$ has an unique point of maximum.
\end{enumerate}
Let $D^{\prime}\subset D$ be the set of points of maximum of the functions belonging to $V$. Then the linear space
\[
V^{\prime}: =\{f|_{D^{\prime}} : f\in V\}
\]
also satisfies (1) and (2).
\end{lemma}

\begin{proof}
We first show that the elements of $V^{\prime}$ satisfy (2). In fact, note
that each nonzero function $f|_{D^{\prime}}\in V^{\prime}$ has a point of maximum
which is unique because $f\in V$ has an unique point of maximum. Now we just need to prove that $\dim(V^{\prime})=n$. Let $\left\{
f_{1},\ldots,f_{n}\right\}  $ be a basis of $V$. If $a_1, \ldots, a_n \in \mathbb{R}$ are such that $\sum_{i=1}^{n}a_{i}f_{i}|_{D^{\prime}}=0,$
then $-\sum_{i=1}^{n}a_{i}f_{i}|_{D^{\prime}}=0$. Supposing that $g: =\sum
_{i=1}^{n}a_{i}f_{i}\neq0;$ it follows that $-g\neq0$. Since the points of maximum
of $g$ and $-g$ belong to $D^{\prime},$ the images of these points of
maximum are $0$, since $g|_{D^{\prime}}=-g|_{D^{\prime}}=0.$ We thus conclude
that $g=0$. Since the set $\left\{  f_{1},\ldots,f_{n}\right\}  $ is linearly
independent, we have $a_{1}=\cdots=a_{n}=0$, and hence $\left\{  f_{1}%
|_{D^{\prime}},\ldots,f_{n+1}|_{D^{\prime}}\right\}  $ is linearly
independent as well. The proof is complete because the inequality $
\dim(V^{\prime})\leq\dim V$ is obvious.
\end{proof}

\begin{lemma}
\label{lemma2} Keeping the terminology and the notation of Lemma \ref{lemmafirst} and that of its proof, consider the continuous function
$$F \colon D\longrightarrow\mathbb{R}^{n}~,~F(y)=(f_{1}(y),\ldots,f_{n}(y)).$$
Let $X:=F(D^{\prime})$. Then

\begin{enumerate}
\item For every $v\in S^{n-1},$ the function $$g_{v} \colon X\longrightarrow\mathbb{R}~,~g_{v}(x)= \langle x,v \rangle,$$ has an unique point of maximum.

\item For every $x\in X$ there is $v\in S^{n-1}$ such that $x$ is the unique point of maximum of the function $g_{v}$.

\item Endowing $D^{\prime}$ with the metric induced by the metric of $D$ and $X$ with the Euclidean metric of $\mathbb{R}^{n}$, then
$F|_{D'}\colon D^{\prime}\longrightarrow X$ is a continuous bijection.
\end{enumerate}
\end{lemma}

\begin{proof}
(1) Given $v=(a_{1},\ldots,a_{n})\in S^{n-1}$, consider the function $g_{v}\circ F\colon D^{\prime
}\longrightarrow\mathbb{R}$. For every $y \in D'$,
\[
g_{v}\circ F(y)= \langle F(y),v \rangle=\sum_{i=1}^{n}a_{i}f_{i}|_{D^{\prime}}(y)\in
V^{\prime}.
\]
From Lemma \ref{lemmafirst}, $g_{v}\circ F \colon D^{\prime}\longrightarrow\mathbb{R}$
has an unique maximum $d\in D^{\prime}$. Therefore, $ \langle F(y),v \rangle \leq \langle F(d),v \rangle$ for every $y\in D^{\prime
}$,
and the equality holds only when $y=d$. Then, $\langle x,v \rangle \leq \langle F(d),v \rangle$  for every $x\in X$. So $F(d)$ is a point of maximum of $g_{v}.$ Assume that $g_{v}$ has another point of
maximum $x^{\prime}$ in $X$. In this case, $x^{\prime}=F(d^{\prime})$ for
some $d^{\prime}\neq d$ and the function $g_{v}\circ F \colon D^{\prime}%
\longrightarrow\mathbb{R}$ has two points of maximum. But this contradicts what we have just proved.

\medskip

\noindent (2) For each $x\in X$ there is
$d\in D^{\prime}$ so that $x=F(d)$. But $d$ is the point of maximum
of some nonnull function $\sum_{i=1}^{n}a_{i}f_{i}|_{D^{\prime}}\colon D^{\prime
}\longrightarrow\mathbb{R}$. With no loss of generality, we may suppose
$\|(a_{1},\ldots,a_{n})\|_2=1$. Calling $v:=(a_{1},\ldots,a_{n})\in S^{n-1}$,  we have,
\[
\sum_{i=1}^{n}a_{i}f_{i}|_{D^{\prime}}(y)= \langle F(y),v \rangle \leq \langle F(d),v \rangle = \langle x,v \rangle,
\] for every $y\in D^{\prime}.$ Then $\langle z, v \rangle \leq \langle x, v \rangle$ for every $z\in X$. From (1) it follows that $x$ is the unique point of maximum of
$g_{v}$.

\medskip

\noindent (3) As each coordinate function $f_{i}|_{D^{\prime
}}$ is continuous and $X=F(D^{\prime})$, we just need to
prove that $F$ is injective. If $d_{1}^{\prime}\neq d_{2}^{\prime
}\in D^{\prime}$ are such that $F(d_{1}^{\prime})=F(d_{2}^{\prime})$ then,
\[
g_{v}\circ F(d_{1}^{\prime})= \langle v,F(d_{1}^{\prime}) \rangle = \langle v,F(d_{2}^{\prime}) \rangle = g_{v}\circ F(d_{2}^{\prime}).
\]
for every $v\in\mathbb{R}^{n}$. Therefore neither $d_{1}^{\prime}$ nor $d_{2}^{\prime}$ can be unique points of
maximum of $g_{v}\circ F\in V^{\prime}$, for any $v$. The proof is complete because this contradicts the
definition of $D^{\prime}$ and Lemma \ref{lemmafirst}.
\end{proof}

\begin{definition}\rm
\label{definitionX} Let $n\geq2$ and let $X$ be a subset of $\mathbb{R}^{n}$ containing more than one point satisfying the following conditions:

\begin{enumerate}
\item For every $v\in S^{n-1}$, the function $$g_{v} \colon X\longrightarrow\mathbb{R}~,~ g_{v}(x)= \langle x,v \rangle,$$
has an unique point of maximum denoted by $x_{v}$.

\item For every $x\in X$ there is $v_{x}\in S^{n-1}$
such that $x$ is the unique point of maximum of the function
$g_{v_{x}}$.
\end{enumerate}
Define the function
$$f\colon S^{n-1}\longrightarrow X~,~f(v)=x_{v}, $$
where $x_{v}$ is described in (1). From (2) it follows that
$f$ is surjective.
\end{definition}

\begin{lemma}
\label{lemmacompact} Let $X$ and $f$ be as in Definition \ref{definitionX}
and let $K$ be a compact subset of $\mathbb{R}^{n}$ so that $K\subset X$. Then
$f^{-1}(K)\subset S^{n-1}$ is a compact subset of $\mathbb{R}^{n}$.
\end{lemma}

\begin{proof}
Let $(v_{n})_{n=1}^\infty $ be a sequence in $f^{-1}(K)$. As $f^{-1}(K)\subset S^{n-1}$ and $S^{n-1}$ is compact, there is a
convergent subsequence $v_{n_{j}}\longrightarrow v\in S^{n-1}$.
For every $j$, let $x_{j}=f(v_{n_{j}})\in K$. Since $K$ is compact, there is a
convergent subsequence $x_{j_{k}}\longrightarrow x\in K.$ Now, since
$x_{j_{k}}$ is the unique point of maximum of the function $ g_{v_{n_{j_{k}}}}$ in $X,$ we have $\langle f(v),v_{n_{j_{k}}} \rangle \leq \langle x_{j_{k}},v_{n_{j_{k}}}\rangle$.  Making $k\longrightarrow\infty$ we get
$ \langle f(v),v \rangle \leq \langle x,v \rangle.$
But $f(v)$ is the unique point of maximum of $g_{v}$, then
$f(v)= x$ and $v\in f^{-1}(K)$. This proves that $f^{-1}(K) \subset S^{n-1}$ is closed, hence compact.
\end{proof}

\begin{lemma}
\label{lemmacontinuous} The function $f$ defined in Definition \ref{definitionX} is continuous if and only if $X$ is compact.
\end{lemma}

\begin{proof}
Recall that in $X$ and $S^{n-1}$ we are considering the Euclidean metric of $\mathbb{R}^{n}$. Suppose that $f \colon S^{n-1}\longrightarrow X$ is continuous. Since $f$ is surjective and
$S^{n-1}$ is compact, it follows that $X=f(S^{n-1})$ is compact. Conversely,  suppose that $X$ is compact. Let $B$ be a closed subset of $X$. Since $X$ is a
compact metric space, we conclude that $B$ is also compact in $\mathbb{R}^{n}%
$. From Lemma \ref{lemmacompact} we know that $f^{-1}(B)$ is a compact subset of $S^{n-1}$, therefore closed.
\end{proof}

\begin{theorem}
\label{kkkp}Let $n\geq2$ and $m\geq1$ be positive integers. Then $m<n$ if, and
only if, for every compact set $K\subset\mathbb{R}^{m}$, there is no $n$-dimensional subspace of $C(K)$ contained in $\widehat{C}(K) \cup \{0\}$. \end{theorem}

\begin{proof} Assume that $m < n$ and suppose that there exist a compact $K\subset\mathbb{R}^{m}$ and a set
$V\subset\widehat{C}(K)$ so that $V\cup\{0\}$ is an $n$-dimensional linear subspace of $C(K)$. The compact $K$ shall play the role of the metric space $D$ in the lemmas above. Let $f_{1},\ldots,f_{n}$ be a basis of $V\cup
\{0\}$ and define $F \colon D\longrightarrow\mathbb{R}^{n}$ as in Lemma \ref{lemma2}. As before, let $D^{\prime}\subset D$ be the set of points of maximum of the functions belonging to $V$ and let $X:=F(D^{\prime})$. Since $\dim(V\cup\{0\})\geq2$, $V$ contains non-constant functions.
Letting $g\in V$ be a non-constant function; it is clear that $g$ and $-g$ have different points of maximum, so $D'$ has more than one point. By Lemma \ref{lemma2} (3), the map $F|_{D'} \colon D^{\prime}\longrightarrow X$ is bijective, thus $X$ has more than one point as well. Therefore $X$ satisfies the conditions of Definition \ref{definitionX}. Consider also the function $f \colon S^{n-1}\longrightarrow X$ from Definition
\ref{definitionX}. Let us prove that $D^{\prime}$ is closed in $D$, and
therefore compact in $\mathbb{R}^{m}$. Let $(d_{k})_{k=1}^\infty$ be a sequence in $D^{\prime}$ converging to $d$. From the compactness of $D$ we have $d\in D$. From Definition \ref{definitionX} we know that, for each $k$,  $F(d_{k})=f(v_{k})$ for some
$v_{k}\in S^{n-1}$, which
means that $F(d_{k})$ is the unique point of maximum in $X$ of the function
\[g_{v_k} \colon X \longrightarrow \mathbb{R}~,~
g_{v_{k}}(x)=\langle x,v_{k}\rangle.
\]
By the compactness of $S^{n-1}$ there is a convergent
subsequence $v_{k_{j}}\longrightarrow v\in S^{n-1}$. Note that
\[
\langle F(d_{k_{j}}),v_{k_{j}}\rangle \geq  \langle F(y),v_{k_{j}} \rangle
{\rm ~for~ every~} y\in D^{\prime}.\]
Since these points $y\in D^{\prime}$ are points of
maximum of the functions $g_{v}\circ F \colon D\longrightarrow\mathbb{R}$, the
inequality
\[
\langle F(d_{k_{j}}),v_{k_{j}}\rangle \geq \langle F(z),v_{k_{j}} \rangle
\]
holds for every $z\in D$. Using the continuity of $F \colon D\longrightarrow\mathbb{R}^{n}$, we get
\[\langle F(d),v \rangle =
\lim_{j\rightarrow\infty} \langle F(d_{k_{j}}),v_{k_{j}}\rangle \geq \lim_{j\rightarrow\infty} \langle F(z),v_{k_{j}} \rangle = \langle F(z), v \rangle,
\]
for every $z\in D$. So $d$ is a point of maximum of the function $g_{v}\circ F \colon D\longrightarrow
\mathbb{R}$, and thus $d\in D^{\prime}$. This completes the proof that $D^{\prime}$ is closed in $D$.

By Lemma \ref{lemma2}(3) we know that the function $F \colon D^{\prime}\longrightarrow X$ is a continuous bijection from the compact $D'$ to the Haussdorf space $X$, therefore it is a homeomorphism. As $D^{\prime}$ is compact, $X=F(D^{\prime})$ is compact as well, and by
Lemma \ref{lemmacontinuous} the function $f \colon S^{n-1}\longrightarrow X$ is continuous. Considering $\mathbb{R}^{m}$ conveniently embedded in $\mathbb{R}^{n-1}$ (remember that $m < n$ by assumption), the function
\[F^{-1}\circ f \colon S^{n-1}\longrightarrow D' \subset \mathbb{R}^m \subset \mathbb{R}^{n-1},
\]
is continuous.

By the Borsuk-Ulam theorem (see, e.g. \cite{Dugundji}) there is a pair of antipodal points $v,-v\in
S^{n-1}$ such that $F^{-1}\circ f(v)=F^{-1}\circ f(-v)$. The injectivity of $F^{-1}$ gives $f(v)=f(-v)=: x$. Using the definition of $f \colon S^{n-1}\longrightarrow X$ we conclude that both
$g_{v} \colon X\longrightarrow\mathbb{R}$ and $g_{-v} \colon X\longrightarrow\mathbb{R}$ have
$x$ as a maximum. On the other hand, $g_{-v}=-g_{v}$, so $g_{v}%
\colon X\longrightarrow\mathbb{R}$ is constant. This is a contradiction because $X$ has more than one point and $g_{v}$ attains its maximum only once in $X$. For the converse just make $D = S^{n-1}$ in Theorem \ref{aaa}.
\end{proof}

\section{An infinite dimensional example}
Let $K$ be a compact subset of $\mathbb{R}^m$. In Section \ref{Sec3} we saw that, for $m < n$, there is no $n$-dimensional subspace of $C(K)$ formed, up to the origin, by functions that attain the maximum only at one point. In this section we show that if we allow $K$ to be a compact subset of an infinite dimensional Banach space, $\widehat{C}(K)$ may contain, up to the origin, an infinite dimensional subspace of $C(K)$.

\begin{example}\rm Let $K$ be the following subset of $\ell_2$:
$$K = \left\{\left(\frac{a_n}{n}\right)_{n=1}^\infty : (a_n)_{n=1}^\infty \in \ell_2 {\rm ~and~} \|(a_n)_{n=1}^\infty\|_2 \leq 1 \right\}. $$
It is clear that $K$ is a subset of the Hilbert cube $ \prod_{n=1}^\infty \left[-\frac{1}{n}, \frac{1}{n} \right]$. Since the Hilbert cube is compact, to prove that $K$ is compact it is enough to show that it is closed. Let $( v_j)_{j=1}^\infty = \left(\left(\frac{v_n^j}{n} \right)_{n=1}^\infty \right)_{j=1}^\infty$ be a sequence in $K$ converging to $w = (w_n)_{n=1}^\infty \in \ell_2$. Since convergence in $\ell_2$ implies coordinatewise convergence, $w_n =\lim\limits_j \frac{v_n^j}{n}$, so $nw_n = \lim\limits_j v_n^j$, for every fixed $n$. For every $k$,
$$\sum_{n=1}^k n^2|w_n|^2 = \sum_{n=1}^k \lim_j|v_n^j|^2 = \lim_j \sum_{n=1}^k |v_n^j|^2 \leq \limsup_j \|(v^j_n)_{n=1}^\infty \|_2^2\leq1.$$
This shows that $\|(nw_n)_{n=1}^\infty \|_2 \leq 1$, proving that $w \in K$. So $K$ is a compact subset of $\ell_2$.

Now we proceed to show that $\widehat{C}(K) \cup \{0\}$ contains an infinite dimensional subspace of $C(K)$. Consider the function
$$F \colon K \longrightarrow \ell_2~,~F\left( \left(\frac{a_n}{n}\right)_{n=1}^\infty \right)= \left( a_n\right)_{n=1}^\infty. $$
By $\pi_j \colon \ell_2 \longrightarrow \mathbb{R}$ we mean the projection onto the $j$-th coordinate, $j \in \mathbb{N}$. For each $j$, the function
$$\pi_j \circ F \colon K \longrightarrow \mathbb{R} $$
is continuous because $\pi_j \circ F = j \cdot \pi_j$. It is clear that the functions $\pi_j \circ F, j \in \mathbb{N}$, are linearly independent. Let $f:=\sum_{j=1}^k b_j(\pi_j \circ F)$ be a nontrivial linear combination of these continuous functions. Writing $b = (b_1, \ldots, b_k,0,0, \ldots) \in \ell_2$ we have $f(x) = \langle b, F(x)\rangle$ for every $x \in K$. As $b \in \ell_2^k$ and $\|F(x)\|_2 \leq 1$ for every $x \in K$,
$$f(x) = \langle b, F(x) \rangle < \left\langle b, \frac{b}{\|b\|_2}\right\rangle $$
whenever $F(x) \neq \frac{b}{\|b\|_2}$. As $F$ is a bijection onto the closed unit ball of $\ell_2$, there is a unique $y \in K$ such that $F(y) = \frac{b}{\|b\|_2}$. This shows that $f$ attains its maximum at $y$. An adaptation of the argument used in the proof of Theorem \ref{bbbb} guarantees that this maximum is unique.
\end{example}

\section*{Appendix}

The result below was communicated by V. I. Gurariy during a Non-linear Analysis Seminar at Kent State University (Kent, Ohio, USA) in the Fall of 2004. The alert reader may hear the echo of Banach's proof of the fact that the set of continuous nowhere differentiable functions is dense and the set of continuous functions differentiable at least at one point is meager.

\medskip

\noindent {\bf Proposition A.} (V.I. Gurariy, 2004) {\it The set $\widehat{\mathcal{C}}[0,1]$ is a $G_{\delta}$ dense subset of $\mathcal{C}[0,1]$.}

\medskip

\noindent{\it Sketch of the proof.}
For each $n \in \mathbb{N}$, let
$$U_n = \left\{ f \in \mathcal C[0,1] : \ {\rm for\ some\ }x \in [0,1], \ f(x) > \max_{|t-x| \geq \frac{1}{n}} f(t) \right\}.$$
Notice that $U_n$ is open. Indeed, if $f \in U_n$ and $g \in \mathcal C[0,1]$ is ``{\em very close}'' to $f$, then
$$g(x) > \max_{|t-x| \geq \frac{1}{n}} g(t).$$
Also, if $h \in \mathcal C[0,1]$ is arbitrary and $h(x_0) = \max_{t \in [0,1]} h(t)$, then (by slightly increasing $h$ near $x_0$) we shall get a function $k \in \mathcal C[0,1]$ such that $h \approx k$ and $k \in U_n.$ In other words, each $U_n$ is dense in $\mathcal C[0,1].$ Furthermore, $\mathcal M = \bigcap_{n=1}^\infty U_n.$ Indeed, suppose that $f \in U_n$ for each $n$ and that $f(x_0) = f(x_1) = \max_{[0,1]} f$ for some $0 \leq x_0 < x_1 \leq 1.$ For large enough $n$, this means that
$$f(x_0) = \max_{[0,1]} f > \max_{|t-x| \geq \frac{1}{n}} f(t) \geq f(x_1) = \max_{[0,1]} f,$$ which is a contradiction. Therefore, $\bigcap_n U_n \subset \mathcal{M}.$ The converse inclusion is easy.

\end{document}